\documentclass[reqno]{amsart}

\usepackage{amsmath}
\usepackage{amsfonts}
\usepackage{amssymb,enumerate}
\usepackage{amsthm}
\usepackage[all]{xy}
\usepackage{rotating}
\usepackage{hyperref}
\usepackage{color}

\theoremstyle{plain}
\newtheorem{lem}{Lemma}[section]
\newtheorem{cor}[lem]{Corollary}
\newtheorem{prop}[lem]{Proposition}
\newtheorem{thm}[lem]{Theorem}

\newtheorem*{mthm*}{Main Theorem}

\theoremstyle{definition}
\newtheorem{defn}[lem]{Definition}

\newtheorem{para}[lem]{}

\newtheorem*{convention*}{Convention}

\newcommand{\id}{\operatorname{id}}

\newcommand{\lotimes}{\otimes^{\mathbf{L}}}
\newcommand{\HH}{\operatorname{H}}

\newcommand{\End}{\operatorname{End}}

\newcommand{\Ker}{\operatorname{Ker}}

\newcommand{\xra}{\xrightarrow}

\newcommand{\y}{\mathbf{y}}

\renewcommand{\geq}{\geqslant}

\newcommand{\Ext}[4][R]{\operatorname{Ext}_{#1}^{#2}(#3,#4)}

\newcommand{\Hom}{\operatorname{Hom}}

\def\Ext{\operatorname{Ext}}

\def\D{\operatorname{\mathsf{D}}}

\def\Diff{\mathrm{Diff}}
\def\E{\mathcal{E}}
\def\D{\mathcal{D}}
\def\Der{\mathrm{Der}}
\def\End{\mathrm{End}}
\newcommand{\grHom}{\operatorname{gr-Hom}}
\def\B{{\mathcal B}}

\numberwithin{equation}{lem}

\begin{document}

\bibliographystyle{amsplain}

\title[The theory of $j$-operators with application]{The theory of $j$-operators with application to (weak) liftings of DG modules}

\author{Saeed Nasseh}
\address{Department of Mathematical Sciences\\
Georgia Southern University\\
Statesboro, GA 30460, U.S.A.}
\email{snasseh@georgiasouthern.edu}

\author{Maiko Ono}
\address{Institute for the Advancement of Higher Education, Okayama University of Science, Ridaicho, Kitaku, Okayama 700-0005, Japan}
\email{ono@pub.ous.ac.jp}

\author{Yuji Yoshino}
\address{Graduate School of Natural Science and Technology, Okayama University, Okayama 700-8530, Japan}
\email{yoshino@math.okayama-u.ac.jp}

\thanks{Y. Yoshino was supported by JSPS Kakenhi Grant 19K03448.}



\begin{abstract}
A major part of this paper is devoted to an in-depth study of $j$-operators and their properties. This study enables us to obtain several results on liftings and weak liftings of DG modules along simple extensions of DG algebras and unify the proofs of the existing results obtained by the authors on these subjects. Finally, we provide a new characterization of the (weak) lifting property of DG modules along simple extensions of DG algebras.
\end{abstract}

\maketitle


\section{Introduction}\label{sec20200314a}

Throughout the paper, $R$ is a commutative ring.\vspace{5pt}

In 1957, Tate~\cite{Tate} showed that the Poincar\'{e} series of every finitely generated module over a complete intersection is rational (with common denominator), and he provided a lower bound for the Betti numbers of a singular local ring. Tate's results on the Poincar\'{e} series and Betti numbers relied on constructing a DG algebra $B$ by adjoining a variable of degree $i>0$ to a DG algebra $A$ in order to kill a cycle of degree $i-1$ in $A$ and repeating this process until the resulting DG algebra is free of cycles. The DG algebra $B$ described as above is called a \emph{simple extension} of the DG algebra $A$; see~\ref{simple extension} for more details. A key point in {\it op.\ cit.} was using the notion of $j$-operators to investigate the induced homomorphism $\HH(A)\to \HH(B)$ of the homology algebras more closely and prove that it is surjective.

The main purpose of this paper is to give the precise definition of $j$-operators, study their properties, and utilize them to obtain several results on liftings and weak liftings of DG modules along simple extensions of DG algebras. We also unify the proofs of the existing results obtained by the authors in~\cite{nassehyoshino,OY}; in this regard, Sections~\ref{sec20200314d} and~\ref{sec20201110a} of this paper can be considered as an addendum to~\cite{nassehyoshino,OY}. It is worth mentioning that, Section~\ref{sec20201110a} contains a new characterization of the (weak) lifting property of DG modules along such DG algebra extensions; see Theorem~\ref{naive thm one variable}.



\section{Preliminaries and conventions}\label{sec20200314b}

This section contains the notation along with the basic facts that are used in the subsequent sections. We assume that the reader is familiar with the category of DG modules over DG algebras. References on this subject include~\cite{avramov:ifr,
avramov:dgha, felix:rht, GL}.

\begin{para}\label{para20200329a}
Throughout the paper,  $A$ is a \emph{strongly commutative differential graded $R$-algebra} (DG $R$-algebra, for short), that is, $A  = \bigoplus  _{n \geq 0} A _n$ is a non-negatively graded $R$-algebra equipped with  a differential map $d^A \colon A \to A(-1)$ (which is a graded $R$-linear map with  $(d^A)^{2} = 0$) such that  $ab = (-1)^{|a| |b|}ba$  for all $a, b \in A$  and  $a^2 =0$  if  $|a|$ is odd, and that $d^A$ satisfies the \emph{Leibniz rule}: for all $a,b\in A$ we have $d^A(ab)=d^A(a)b+(-1)^{|a|}ad^A(b)$. Here, $A(-1)$ denotes the $-1$st suspension of $A$ and $|a|$ denotes the degree of the homogeneous element $a\in A$.
\end{para}

\begin{para}\label{simple extension}
In this paper, let  $B = A \langle X \mid dX = t \rangle$  be a DG $R$-algebra with the differential $d:=d^B$ that is a simple extension of the DG $R$-algebra $A$ obtained by adjunction of a variable  $X$ with $|X| > 0$ such that $dX = t \in A$  is a cycle. We will simply denote $B$ by $A\langle X \rangle$.
Note that  $|X| = |t|+1$ and we have the following cases.
\begin{enumerate}[\rm(a)]
\item
If $|X|$ is odd, then  $X^{2} = 0$ and  $B = A \oplus XA$. Moreover, for all $a + Xb\in B$
$$
d(a + Xb)=d^A(a) + tb - X d^A(b).
$$

\item
If $|X|$ is even, then  $B$ is a free algebra with divided power,
that is,  $B = A \oplus XA \oplus X^{(2)}A \oplus X^{(3)} A \oplus \cdots $
with the differential structure $dX^{(m)}=X^{(m-1)}t$  and with the algebra structure $X^{(m)}X^{(\ell)} =\binom{m+\ell}{m} X^{(m+\ell)}$, considering the conventions $X^{(0)}=1$  and $X^{(1)}=X$. Note that the action of $d$ on a general element of $B$ is given by the formula
\begin{equation}\label{eq20200315v}
d\left(\sum_{i=0}^nX^{(i)}a_i\right)=\sum_{i=0}^{n-1}X^{(i)}\left(d^A(a_i)+ta_{i+1}\right)+X^{(n)}d^A(a_n).
\end{equation}
\end{enumerate}
\end{para}

\begin{para}\label{para20200329b}
Throughout the paper, $N$ is a graded free  right  $B$-module, that is,
$N = \bigoplus _{\lambda \in \Lambda}  e _{\lambda} B$,  where  $\Lambda$ is an index set and we have
$e _{\lambda} B \cong B (-|e_{\lambda}|)$  as graded $B$-modules for each $\lambda \in \Lambda$.
Note that since  $B$ is graded commutative, $N$ is also a graded left $B$-module with the left $B$-action
$bx = (-1)^{|x||b|} xb$  for  $b\in B$ and $x \in N$.

The \emph{graded endomorphism ring} of $N$, denoted $\E$, is
$$
\E := \End ^*_B(N) = \bigoplus _{n \in \mathbb{Z}} \grHom_B(N, N(n)).
$$
Note that an $R$-homomorphism  $f\colon N \to N(n)$  is $B$-linear if  $f(xb)= f(x)b$ for $x \in N$ and $b \in B$.
Note also that $\E$ is not only an associative graded ring but also a Lie subalgebra of $\End ^* _R (N)$  by defining the bracket product as
$$
[f, g] = f\circ g - (-1)^{|f||g|}g\circ f.
$$
From now on, for simplicity, we may denote the composition notation $f\circ g$ by $fg$. We will specify the composition notation if there is a fear of confusion.

For homogeneous elements $f,g,h\in \E$, the \emph{Jacobi identity} is
$$
[[f,g],h]=[f,[g,h]]-(-1)^{|f||g|}[g,[f,h]].
$$
\end{para}

\begin{para}\label{para20200314b}
The left multiplication map by an element  $b \in B$, denoted $\ell _b$, is an element of  $\E$; in fact, $\ell_b(x)=bx$ for all $x\in N$. In other words, we have a map $B \longrightarrow \E$ that is given by $b \mapsto \ell _b$.
By identifying  $b$ with  $\ell _b$ under this map, we regard  $B$  as a  subring of  $\E$. Note that for all $a,b\in B$ we have $[\ell_a,\ell_b]=0$.

If $f\colon N \to N(n)$ is an additive map, then for all $b\in B$ and $x\in N$ we have
$f \ell_b (x) = f((-1)^{|b||x|}xb) = (-1)^{|b||x|}f(xb)$ and
$\ell_b f (x) = (-1)^{(|f|+|x|)|b|}f(x)b$.
Hence, we have
\begin{equation}\label{eq20200314a}
[f, \ell_b](x) = (-1)^{|x||b|}(f(xb)-f(x)b).
\end{equation}
This implies that
$[f, \ell_b] =0$ if and only if $f$ is $B$-linear (that is, $f\in \E$).
\end{para}

\begin{para}\label{para20200314c}
Let $\Diff _B (N)$  be the set of all $R$-linear maps $\partial\colon N \to N (-1)$ that satisfy the \emph{Leibniz rule}, that is, $\partial(xb) = \partial(x) b + (-1)^{|x|} x db$ for $x \in N$ and $b\in B$.

By~\eqref{eq20200314a}, for an $R$-linear map $\partial \colon N \to N (-1)$ and for all $b\in B$ and $x\in N$ we have
$[\partial , \ell_b](x) = (-1)^{|b||x|} (\partial (xb) - \partial (x)b)$. We also have $\ell _{db} (x) = (-1) ^{(|b|-1)|x|}x db$.
Thus, $\partial$ belongs to  $\Diff_B(N)$  if and only if  $[\partial, \ell_b] =\ell _{db}$  for all $b \in B$.

We call an element of $\Diff_B(N)$ a \emph{differential on $N$}.
If  $\partial \in \Diff _B(N)$  satisfies  $\partial ^2= 0$,  then  $(N,  \partial)$ defines a right DG $B$-module structure.

Note that if $0$ is in $\Diff_B(N)$, then $\ell_{db}=0$ for all $b \in B$ and since  $N$  is a free $B$-module, it means that $d=0$. Therefore, if $d \not= 0$ then  $0$ is not in $\Diff_B(N)$. Using~\ref{para20200314b}, by the same argument we have that if $d\not=0$, then $\E \cap \Diff _B(N) = \emptyset$.
\end{para}

\begin{para}\label{para20200315g}
For a given basis  $\mathcal{B} = \{ e_{\lambda} \}$ of $N$ as a free $B$-module, we define the \emph{free differential}  $\partial ^{\mathcal{B}} \in \Diff _B(N)$  by  $\partial ^{\mathcal{B}}  (e_{\lambda}) = 0$ for all $\lambda$.
Recall that $N\cong \bigoplus_{\lambda}B(-|e_{\lambda}|)$ and note that  $\partial  ^{\mathcal{B}} (\sum _{\lambda} e_{\lambda} b_{\lambda} ) = \sum _{\lambda} (-1)^{|e_{\lambda}|}e_{\lambda} db_{\lambda}$. In particular, $(N, \partial ^{\mathcal{B}})$ is a free DG $B$-module.

Note also that  if  $\partial, \partial ' \in \Diff _B(N)$  then  $\partial - \partial ' \in \E$.
Therefore,  every differential on $N$ can be written as  $\partial ^{\mathcal{B}}+ g$  for some $g \in \E$.
\end{para}

In the following we always assume that  $d \not=0$; see~\ref{para20200314c}.

\begin{para}\label{Jacobi}
$\E \oplus \Diff _B(N)$ is a Lie subalgebra  of  $\End _R^* (N)$.
In fact, the following holds:
$$[\E \oplus \Diff _B(N), \E \oplus \Diff _B(N)] \subseteq \E.$$
To see this, let  $f, f' \in \E$ and $\partial, \partial ' \in \Diff _B (N)$.
The fact that $[f, f'] \in \E$  is clear. For all $b\in B$, note that $[f, \ell_b] = 0$. Hence, from Jacobi identity we have $[[f, \partial], \ell_b]  = [f,[\partial,  \ell_b]]=[f,\ell_{db}] = 0$.
Therefore, $[f, \partial] \in \E$.
Similarly,  $[[\partial, \partial'], \ell_b] = [\partial, \ell_{db}] +[\partial ', \ell_{db}] = \ell_{d^2b} + \ell_{d^2b}=0$, and thus, $[\partial, \partial '] \in \E$.
\end{para}

\begin{para}\label{para20200315a}
An $R$-linear map $\varphi\colon N \to N$ is called an \emph{at most first-order derivation on $N$} if $[ \varphi , \ell_b] \in \E$ for all $b \in B$, equivalently, $[[\varphi, \ell_b], \ell_{b'}]=0$ for all $b, b' \in B$.
All elements in $\E \cup \Diff _B(N)$ are at most first-order derivations.
\end{para}

\begin{para}\label{def D} Let $\D$ be the set of all $R$-linear maps $\varphi\colon N \to N$ for which there is $g_{\varphi} \in \E$ (depending only on $\varphi$) such that $[\varphi, \ell_b]= g_{\varphi}\ell_{db}$, for all $b \in B$.

Since $g_{\varphi}$ is $B$-linear, it follows from~\ref{para20200314b} and~\ref{para20200315a} that each element $\varphi\in \D$ is an at most first-order derivation. Also  $\D$ is an abelian group, i.e., if $\varphi, \psi \in \D$, then $\varphi  \pm \psi \in \D$.

Note that $\E \cup \Diff_B(N)  \subseteq \D$.
In fact, if $\varphi \in \E$ (resp. $\varphi \in \Diff_B(N)$)  then $g _{\varphi}=0$ (resp. $g _{\varphi}= 1$, the identity map).
Hence, $\E \oplus \Diff _B(N) \subseteq \D$.

Note also that  $\D$  is a graded abelian group;  $\varphi \in \D$ is homogeneous of degree $|\varphi |$  if and only if   $\varphi$ sends $N _i$  into  $N _{i+|\varphi |}$ for all $i \in \mathbb{Z}$.
We have $| \varphi | = |g _{\varphi}| -1$ if  $g_{\varphi} \not=0$.
\end{para}

\begin{para}\label{formula1}
Let  $f \in \E$, $\partial  \in \Diff _B(N)$, and $b \in B$. Then
$[f  \partial, \ell_b] = f \ell _{db}$ and $[\partial f, \ell_b]= (-1)^{|f|} f \ell _{db}$.
To see the first equality, note that by~\ref{para20200314b} and~\ref{para20200314c} we have
\begin{eqnarray*}
[f  \partial, \ell_b] &=& f \partial \ell _b - (-1)^{(|f|-1)|b|} \ell_b f \partial  \\
&=&f \partial \ell _b - (-1)^{|b|} f \ell_b \partial + (-1)^{|b|} f \ell_b \partial  - (-1)^{(|f|-1)|b|} \ell_b f \partial \\
&=& f [\partial , \ell_b] + (-1)^{|b|}[f, \ell_b] \partial\\
& =& f \ell _{db}.
\end{eqnarray*}
Similarly, for the second equality we have
\begin{eqnarray*}
[\partial f , \ell_b] &=& \partial f \ell _b - (-1)^{(|f|-1)|b|} \ell_b \partial  f \\
&=& \partial f \ell _b - (-1)^{|b||f|}  \partial \ell_b f + (-1)^{|b||f|}  \partial  \ell_b f - (-1)^{(|f|-1)|b|} \ell_b  \partial f \\
&=& \partial [f , \ell_b] + (-1)^{|b||f|}[\partial , \ell_b] f\\
&  =&  (-1)^{|b||f|} \ell _{db} f \\
& =& (-1)^{|b||f|+|db||f|} f \ell _{db}\\
& =& (-1)^{|f|} f \ell _{db}.
\end{eqnarray*}
Therefore, if  $f \in \E$ and $\partial  \in \Diff_B(N)$, then  the composition maps $f \partial$  and  $\partial f$ from $N$ to itself belong to $\D$.
In other words, $\E \circ \Diff _B(N) + \Diff _B(N) \circ \E \subseteq \D$.
\end{para}


\begin{para}\label{prop20200315a}
Let $\varphi \in \D$, and let  $\partial \in \Diff _B(N)$ be an arbitrary differential.
There is  $g=g_{\varphi} \in \E$ such that  $[\varphi , \ell_b] = g \ell_{db}$ for all $b \in B$.
On the other hand, \ref{formula1} shows that $[g \partial, \ell_b]  = g \ell _{db}$.
Hence, $[\varphi - g \partial , \ell_b] =0$ for all $b \in B$.
Thus,  $\varphi - g \partial \in \E$.
Similarly, $\varphi -(-1)^{|g|} \partial g \in \E$.
Therefore, every element of $\D$ is of the form
$f + g \partial$ and  $f' + \partial g'$  for some $f, f', g, g' \in \E$.
In particular, we have the equalities $$\D = \E + \E \circ \partial  =  \partial \circ \E +\E$$ and hence, $\D$  is a two-sided  $\E$-module generated by $\{1, \partial\}$.
We also have $$\D = \E + \E \circ \Diff _B(N) =  \Diff _B(N) \circ \E +\E.$$
\end{para}

\begin{para}\label{yuji20201115a}
Note that for every $\partial \in \Diff _B(N)$ we have $\partial ^2 = \partial \circ \partial \in \E$. This follows from~\ref{para20200314b}, \ref{para20200314c}, and the fact that for all $b\in B$ we have
$$[\partial ^2, \ell _b] = \partial [\partial, \ell _b] + (-1)^{|b|}[\partial, \ell _b] \partial = \partial \ell_{db} + (-1)^{|b|}\ell_{db} \partial = [\partial, \ell_{db}] = \ell _{d^2b} =0.$$
\end{para}

\begin{para}\label{para20201117a}
Note that $\D\subset \End^*_R(N)$ is an $R$-subalgebra. To see this, since $\D \subset \End^*_R(N)$ is an $R$-submodule, it is enough to show that $\D \circ \D \subseteq \D$.
By \ref{prop20200315a} we know that $\D \circ \E \subseteq \D$. Hence, it suffices to show that $\D \circ \partial \subseteq \D$ for some $\partial \in \Diff _B(N)$.
This follows from \ref{yuji20201115a},  since
$\D \circ \partial = (\E + \E \circ \partial)\circ \partial = \E \circ \partial + \E \circ \partial ^2 \subseteq \E \circ \partial + \E = \D$.
\end{para}

\begin{para}\label{para20200320d}
For $f \in \E$ and $g\in \D$ we define $ad(f)(g)=[f, g]$. It follows from~\ref{prop20200315a} that $ad (f)$ is a map from $\D$ to itself.
\end{para}

\section{$j$-operators}\label{sec20200314c}

As we mentioned in the introduction, the notion of $j$-operators was introduced by Tate~\cite{Tate}. The purpose of this section is to give the precise definition of $j$-operators and list some of their properties for later use in the subsequent sections.

In this section we use the notation from Section~\ref{sec20200314b}. To distinguish between the differential of $B$ and the derivative operator (see~\ref{para20200315d} below), we specifically denote the differential of $B$ by $d^B$.

\begin{para}\label{para20200315d}
The \emph{derivative operator}, denoted $\frac{d}{dX}\colon B \to B$,  is defined as follows:

If $|X|$ is odd and  $b = a_0 + Xa_1\in B$ for  $a_0, a_1 \in A$, then  $\frac{db}{dX} = a_1$.

If  $|X|$ is even and $b = \sum_{i=0}^n X^{(i)} a_i\in B$ for $a_i\in A$, then
$\frac{db}{dX} = \sum_{i=1}^n X^{(i-1)} a_i$.
\end{para}

\begin{lem}\label{d/dX}
The following equalities hold for  $b, b' \in B$:
\begin{gather}
 \label{20200315x}
 \frac{d}{dX} (b b') = \frac{db}{dX} \ b' + (-1) ^{|b||X|} b \  \frac{db'}{dX}\\
 \label{20200315y}
\frac{d}{dX} (d^B b) = (-1)^{|X|} d^B (\frac{db}{dX}).
\end{gather}
Also, $\frac{db}{dX}=0$  if and only if  $b \in A$.
\end{lem}

\begin{proof}
Assume first that $|X|$ is odd. Let $b=a_0+Xa_1$ and $b'=a'_0+Xa'_1$ be elements in $B$ with $a_0, a_1, a'_0, a'_1\in A$ . Then
\begin{eqnarray*}
\frac{db}{dX} \ b' + (-1) ^{|b||X|} b \  \frac{db'}{dX}
&=&a_1(a_0'+Xa_1')+ (-1) ^{|a_0|}(a_0+Xa_1)a_1' \\
&=&a_1a'_0 + (-1)^{|a_0|} a_0 a_1' + (-1)^{|a_1|} Xa_1a_1' +   (-1) ^{|a_0|}Xa_1a_1' \\
&=&a_1a'_0 + (-1)^{|a_0|} a_0 a_1'\\
&=&\frac{d}{dX} \left(bb'\right)
\end{eqnarray*}
which is~\eqref{20200315x}. Also, for~\eqref{20200315y} we have
$$
\frac{d}{dX} \left( d^B b \right) = -d^Aa_1
 = (-1)^{|X|} d^B \left(\frac{d}{dX}(a_0+ Xa_1)\right)=(-1)^{|X|} d^B (\frac{db}{dX}).
$$

For the rest of the proof assume that $|X|$ is even. Let $b=\sum_{i=0}^nX^{(i)}a_i$ and $b'=\sum_{i=0}^mX^{(i)}a'_i$ with $a_i, a'_i\in A$ be elements in $B$. To prove~\eqref{20200315x}, we proceed by induction on $n$. The base case is obtained from the following (with $n=0$).

If $b=X^{(n)}a$ with $a\in A$ consists of only one term, then
\begin{eqnarray*}
\frac{db}{dX} \ b' + b \  \frac{db'}{dX}&=& \left(X^{(n-1)}a\right)\left(\sum_{i=0}^mX^{(i)}a'_i\right)+\left(X^{(n)}a\right)\left(\sum_{i=1}^mX^{(i-1)}a'_i\right)\\
&=&\sum_{i=0}^m \binom{n+i-1}{n-1}X^{(n+i-1)}aa'_i+\sum_{i=1}^m \binom{n+i-1}{n}X^{(n+i-1)}aa'_i\\
&=& \sum_{i=0}^m \binom{n+i}{n}X^{(n+i-1)}aa'_i\\
&=&\frac{d}{dX}(bb').
\end{eqnarray*}
For the general case, note that $b=b_1+X^{(n)}a_n$, where $b_1=\sum_{i=0}^{n-1}X^{(i)}a_i$. We then have $bb'=b_1b'+X^{(n)}a_nb'$. The second equality in the next display follows from the previous case and inductive step:
\begin{eqnarray*}
\frac{d}{dX} (bb')&=& \frac{d}{dX} (b_1b')+\frac{d}{dX} \left(X^{(n)}a_nb'\right)\\
&=&\frac{db_1}{dX} \ b' + b_1 \  \frac{db'}{dX}+\frac{d\left(X^{(n)}a_n\right)}{dX} \ b' + X^{(n)}a_n \  \frac{db'}{dX}\\
&=& \frac{db_1}{dX} \ b' + b_1 \  \frac{db'}{dX}+X^{(n-1)}a_n \ b' + X^{(n)}a_n \  \frac{db'}{dX}\\
&=& \left(\frac{db_1}{dX}+X^{(n-1)}a_n\right)b'+\left(b_1+X^{(n)}a_n\right)\ \frac{db'}{dX}\\
&=&\frac{db}{dX} \ b' + b \  \frac{db'}{dX}.
\end{eqnarray*}
Equality~\eqref{20200315y} follows directly from~\eqref{eq20200315v}.
\end{proof}

Here is the definition of \emph{restricted $j$-operators} defined on $\E \cup \Diff _B(N)$. The general definition of $j$-operators defined on $\D$ will be given in Definition~\ref{def j-op}.

\begin{para}\label{para20201110a}
Let  $\mathcal{B} = \{ e_{\lambda}\}$ be a basis of $N$ as a free $B$-module.
We define $$j _X^{\mathcal{B}}\colon \E \cup \Diff _B(N) \longrightarrow \E$$ as follows:
for any $\alpha \in \E \cup \Diff _B(N)$  write  $\alpha ( e_{\lambda}) = \sum _{\mu} e_{\mu} b _{\mu \lambda}$ where  $b _{\mu \lambda} \in B$.
Then $j_X^{\mathcal{B}}(\alpha)$  is defined to be a $B$-linear map that acts on a basis element  $e_{\lambda}$ by
\begin{equation}\label{def j-op}
j_X^{\mathcal{B}}(\alpha) ( e_{\lambda}) = \sum _{\mu} (-1) ^{|e_{\mu}||X|} e_{\mu} \frac{d b _{\mu \lambda}}{dX}.
\end{equation}
\end{para}

Note that $j_X^{\mathcal{B}} ( \partial ^{\mathcal{B}}) = 0$,  where  $\partial ^{\mathcal{B}}$ is the free differential associated to the basis $\mathcal{B}$ defined in~\ref{para20200315g}.
Note also that  $|j_X^{\mathcal{B}}| = -|X|$.

\begin{lem}\label{LR for E}
Let  $f, g \in \E$.
Then, for $j =   j _X^{\mathcal{B}}$ we have
$$
j(fg) =  j(f) g + (-1)^{|X||f|} f j(g).
$$
\end{lem}

\begin{proof}
Note that both sides are $B$-linear transformations of $N$, so to prove the equality it suffices to show that these maps have identical images for a basis element $e_{\lambda}\in \mathcal B$. To see this, let
$f (e_{\lambda}) = \sum _{\mu} e_{\mu} F _{\mu \lambda}$ and $g (e_{\lambda}) = \sum _{\mu} e_{\mu} G_{\mu \lambda}$ with $F_{\mu \lambda},  G_{\mu \lambda}\in B$.
Then
$f g (e_{\lambda}) = \sum _{\nu} e_{\nu} \sum _{\mu} F _{\nu \mu} G_{\mu \lambda}$ and it follows from Lemma \ref{d/dX} that
\begin{eqnarray*}
j(fg) (e_{\lambda} ) &=& \sum _{\nu} (-1)^{|e_{\nu}||X|} e_{\nu}  \frac{d}{dX} \left( \sum _{\mu} F _{\nu \mu} G_{\mu \lambda} \right) \\
&=&  \sum _{\nu} (-1)^{|e_{\nu}||X|} e_{\nu}  \left( \sum _{\mu} \frac{dF _{\nu \mu}}{dX}  G_{\mu \lambda} + \sum _{\mu} (-1) ^{|F _{\nu \mu}||X|} F _{\nu \mu} \frac{dG _{\nu \mu}}{dX}  \right)  \\
&=&  \sum _{\nu, \mu} (-1)^{|e_{\nu}||X|} e_{\nu}   \frac{dF _{\nu \mu}}{dX}  G_{\mu \lambda} + \sum _{\nu, \mu} (-1) ^{(|f| + |e_{\mu}|) |X|} e_{\nu}  F _{\nu \mu} \frac{dG _{\nu \mu}}{dX}  \\
&=&  j(f)g(e_{\lambda})
+ (-1) ^{|X||f|} f j(g) (e_{\lambda}).
\end{eqnarray*}
The third equality uses the fact that $|F_{\nu \mu}| + |e_{\nu}| = |f| + |e_{\mu}|$.
\end{proof}

\begin{prop}\label{LR for general}
Let  $f, g, h \in \E$ and  $\partial \in \Diff_B(N)$.
Assume that  $f + g \partial + \partial h =0$  as an element of  $\D$.
Then, for $j =   j _X^{\mathcal{B}}$ we have
\begin{equation}\label{eq20200320a}
j(f) + j(g) \partial +(-1)^{|X||g|}g j(\partial )+ j(\partial ) h + (-1)^{|X|} \partial j(h) =0.
\end{equation}
\end{prop}

Note that~\eqref{eq20200320a} is formally obtained by applying the Leibniz rule for $j$ to $f + g \partial + \partial h =0$.

\begin{proof}
Note from the definition of restricted $j$-operators~\ref{para20201110a} that  $j( \alpha + \beta ) = j(\alpha ) + j(\beta)$ for all $\alpha, \beta \in \E \cup \Diff_B(N)$.
Recall from~\ref{para20200315g} that $\partial$ can be written as  $\partial = \partial ^{\mathcal{B}} + q$ for some $q \in \E$.
Therefore, by Lemma~\ref{LR for E}, it is enough to prove~\eqref{eq20200320a} only in the case that  $\partial = \partial ^{\mathcal{B}}$ for the basis  $\mathcal{B} = \{e_{\lambda}\}$ of $N$.
In this situation, since  $j( \partial ) =0$, we only need to show $j(f) + j(g) \partial + (-1)^{|X|} \partial j(h) =0$.

Let $f(e_{\lambda} ) = \sum _{\mu} e_{\mu} F_{\mu \lambda}$,  $g(e_{\lambda} ) = \sum _{\mu} e_{\mu} G_{\mu \lambda}$, and $h(e_{\lambda} ) = \sum _{\mu} e_{\mu} H_{\mu \lambda}$,
where
$ F_{\mu \lambda},   G_{\mu \lambda},  H_{\mu \lambda}  \in B$.
From~\ref{para20200315g} we know that
$\partial (\sum _{\lambda} e_{\lambda} b_{\lambda})= \sum _{\lambda} (-1) ^{|e_{\lambda }|} e_{\lambda}  d^Bb_{\lambda}$. Hence,
we have
\begin{eqnarray*}
0\!\!\!\!\! &=& \!\!\!\!\!(f+g \partial + \partial h  ) \left( \sum _{\lambda} e_{\lambda} b_{\lambda}\right) \\
 &=&\!\!\!\!\! \sum_{\mu, \lambda} e_{\mu} F_{\mu \lambda} b_{\lambda} +  \sum_{\mu, \lambda} (-1) ^{|e_{\lambda} | } e_{\mu}  G_{\mu \lambda} d^Bb_{\lambda} + \sum_{\mu, \lambda} (-1) ^{|e_{\mu} | } e_{\mu} d^B(H_{\mu \lambda} b_{\lambda} ) \\
 &=& \!\!\!\!\!\sum_{\mu}\!\! e_{\mu}\!\!\sum _{\lambda}\! \left[\! F_{\mu \lambda} b_{\lambda} + (-1) ^{|e_{\lambda} | } G_{\mu \lambda} d^Bb_{\lambda}  + (-1) ^{|e_{\mu} | }\!\!
 \left(\!d^BH_{\mu \lambda}b_{\lambda} +(-1)^{|H_{\mu \lambda}|} H_{\mu \lambda}d^Bb_{\lambda}\!\right)\!\right].
\end{eqnarray*}
This implies that
\begin{equation*}\label{1.0}
\!\!\sum _{\lambda} \left[\left(F_{\mu \lambda}+(-1) ^{|e_{\mu} | } d^BH_{\mu \lambda}\right) b_{\lambda} +
\left( (-1) ^{|e_{\lambda} | }G_{\mu \lambda} +  (-1)^{|H_{\mu \lambda}|+|e_{\mu} | }H_{\mu \lambda}\right) d^Bb_{\lambda}\right]  = 0
\end{equation*}
which holds for any $b_{\lambda} \in B$ and any $\mu$.
Setting  $b_{\lambda} =1$ if $\lambda=\nu$ and $b_{\lambda}=0$ if $\lambda\neq \nu$,
for all $\mu, \nu$ this equation implies that
\begin{equation}\label{1.1}
F_{\mu \nu} +  (-1) ^{|e_{\mu} | }d^BH_{\mu \nu} = 0.
\end{equation}
Thus, $\sum _{\lambda}\left((-1) ^{|e_{\lambda} | }G_{\mu \lambda} +(-1)^{|H_{\mu \lambda}|+|e_{\mu} | }H_{\mu \lambda}\right)d^Bb_{\lambda}  = 0$.
Since $|H_{\mu \lambda}| + |e_{\mu}| = |h| + |e_{\lambda}|$, we conclude from this equality that
\begin{equation}\label{1.2}
\left(G_{\mu \nu} + (-1)^{|h |} H_{\mu \nu}\right) d^Bb = 0
\end{equation}
for all $b\in B$ and $\mu, \nu$.
Now
\begin{eqnarray*}
& &\left( j(f) + j(g) \partial + (-1)^{|X|} \partial j(h) \right)  \left(\sum _{\lambda} e_{\lambda} b_{\lambda}\right)  \\
&=& \sum _{\mu, \lambda} (-1) ^{|e_{\mu} | |X|} e_{\mu} \frac{dF_{\mu \lambda}}{dX} b_{\lambda }+
\sum _{\mu, \lambda} (-1) ^{|e_{\mu} | |X|+|e_{\lambda}|} e_{\mu} \frac{dG_{\mu \lambda}}{dX} d^Bb_{\lambda }\\
&+&( -1)^{|X|}
 \sum _{\mu,\lambda} (-1) ^{|e_{\mu} | |X| + |e_{\mu}|} e_{\mu}  d^B\left( \frac{dH_{\mu \lambda}}{dX} b_{\lambda }\right) \\
 &=& \sum _{\mu, \lambda} (-1) ^{|e_{\mu} | |X|}  e_{\mu}
\left(  \frac{dF_{\mu \lambda}}{dX}+(-1)^{|X| + |e_{\mu}|} d^B\left( \frac{dH_{\mu \lambda}}{dX}\right) \right) b_{\lambda }  \\
&+ & \sum _{\mu, \lambda} (-1) ^{|e_{\mu} | |X|+|e_{\lambda}|} e_{\mu}  \left(  \frac{dG_{\mu \lambda}}{dX}+(-1)^{|H_{\mu \lambda}|  + |e_{\mu}| -|e_{\lambda}|}\frac{dH_{\mu \lambda}}{dX} \right) d^Bb_{\lambda }.
 \end{eqnarray*}
Note that from~\eqref{20200315y} and~\eqref{1.1} we have
$$
\frac{dF_{\mu \lambda}}{dX}+(-1)^{|X|+|e_{\mu}|}d^B\left( \frac{dH_{\mu \lambda}}{dX} \right)= 0.
$$
Also it follows from~\eqref{20200315x} and~\eqref{1.2} that for all $b \in B$
$$
\left( \frac{dG_{\mu \lambda}}{dX}+(-1)^{|h|}\frac{dH_{\mu \lambda}}{dX} \right) d^Bb = 0.
$$
Thus,  $j(f) + j(g) \partial + (-1)^{|X|} \partial j(h) = 0$, as desired.
\end{proof}

Paragraph~\ref{prop20200315a} enables us to give a general description of $j$-operators.

\begin{cor}\label{cor20200316a}
The restricted $j$-operator  $j_X^{\mathcal B}\colon \E \cup \Diff _B(N) \to \E$ from~\ref{para20201110a} can be extended to  a mapping $\D \to \D$, which we also denote by $j  = j _X^{\mathcal B}$, as follows:

For all $f + \partial g \in \D= \E + \partial\circ \E$, let  $j (f + \partial g ) = j(f) + j ( \partial) g + (-1) ^{|X|} \partial j(g)$  and for all $f'+ g' \partial  \in \D=\E + \E \circ \partial$, let $j(f'+ g'\partial)=j(f') +j(g')  \partial + (-1) ^{|X||g'|} g' j( \partial)$.

Also, if $f + \partial g= f'+ g' \partial$ in $\D$, then $j(f + \partial g)=j(f'+ g' \partial)$. In other words, for every element $\varphi\in \D$, the element $j(\varphi)$ is independent of the description of $\varphi$.
\end{cor}

\begin{proof}
For $\varphi \in \D$, uniqueness of $j(\varphi)$ follows from Proposition~\ref{LR for general}.
\end{proof}

\begin{defn} \label{def j-op}
The mapping  $j_X^{\mathcal{B}}\colon \D \to \D$ defined in Corollary~\ref{cor20200316a} is called the \emph{$j$-operator of $X$ with respect to the basis $\mathcal{B}$} for $N$ as a free $B$-module.
\end{defn}

Note that the $j$-operator depends on the choice of a basis and a variable $X$.

\begin{lem}\label{j-operator properties}
The $j$-operator  $j =  j _X^{\mathcal{B}}\colon \D\to \D$  satisfies the following properties:
\begin{enumerate}[\rm(a)]
\item  $j$ is a graded mapping of degree $-|X|$, that is,  $|j(\alpha )| = |\alpha | -|X|$ for  a homogeneous $\alpha  \in \D$.
\item  $j (X) = 1$ and  $j(a)=0$  for $a \in A$.
\item  If $\partial \in \Diff _B(N)$ with $\partial ^2 =0$, then  $[j(\partial), \partial] =0$.
\end{enumerate}
\end{lem}

 \begin{proof}
(a) is obvious from the definition.

For (b) note that $X\in B$ defines a mapping $X\colon N\to N$ of degree $|X|$ in $\E$ that is given by multiplication from left by $X$, i.e., for all $n\in N$ define $X(n)=Xn=(-1)^{|X||n|}nX$. In particular, for any basis element $e_{\lambda}\in N$ we have $X(e_{\lambda})=(-1)^{|X||e_{\lambda}|}e_{\lambda}X$. By definition of $j$-operator we have the equality $j(X)(e_{\lambda})=(-1)^{|X||e_{\lambda}|+|X||e_{\lambda}|}e_{\lambda}(dX/dX)=e_{\lambda}$. This justifies the equality $j (X) = 1$. The equality $j(a)=0$  for $a \in A$ is checked similarly.

(c) We know from~\ref{Jacobi} that  $[j(\partial), \partial] \in \E$. Thus, to prove part (c), it suffices to show $[j(\partial), \partial] (e_{\lambda})=0$ for all basis elements $e_{\lambda}\in N$.

Let  $\partial (e_{\lambda}) = \sum _{\mu} e_{\mu}b_{\mu \lambda}$. Then
$j(\partial) (e_{\lambda}) = \sum _{\mu} (-1)^{|e_{\mu}||X|} e_{\mu}(d b_{\mu \lambda}/dX)$. Hence,
\begin{gather*}
\partial j(\partial) (e_{\lambda}) = \sum _{\nu, \mu} (-1)^{|e_{\mu}||X|} e_{\nu} b_{\nu \mu} \frac{d b_{\mu \lambda}}{dX} + \sum _{\mu} (-1)^{|e_{\mu}||X| + |e_{\mu}|} e_{\mu}   d^B\left(\frac{d b_{\mu \lambda}}{dX}\right) \\
j(\partial)\partial (e_{\lambda}) =\sum _{\nu, \mu} (-1)^{|e_{\nu}||X|} e_{\nu} \frac{d b_{\nu \mu}}{dX}b_{\mu \lambda}.
\end{gather*}
On the other hand, $\partial ^2=0$  means that $\sum _{\nu, \mu} e_{\nu} b_{\nu \mu} b_{\mu \lambda} + \sum _{\mu} (-1) ^{|e_{\mu}|} e_{\mu} d^Bb_{\mu \lambda} =0$.
Hence, $\sum _{\mu} b_{\nu \mu} b_{\mu \lambda} + (-1) ^{|e_{\nu}|} d^Bb_{\nu \lambda} =0$
for all $\nu, \lambda$.
Taking  $d/dX$ of this equation, by Lemma~\ref{d/dX} we have
$$
\sum _{\mu} \frac{db_{\nu \mu}}{dX}  b_{\mu \lambda} + \sum _{\mu} (-1) ^{|b_{\nu \mu}||X|} b_{\nu \mu}\frac{db_{\mu \lambda}}{dX}  + (-1) ^{|e_{\nu}|+|X|}  d^B\left(\frac{db_{\nu \lambda}}{dX}\right) =0.
$$
Note  that  $|j(\partial)| = -|X|-1$ and
$|b_{\nu \mu}| = |e_{\mu}| -|e_{\nu}|-1$. Therefore, we have
\begin{eqnarray*}
[j(\partial), \partial] (e_{\lambda}) &=&j(\partial) \partial (e_{\lambda}) - (-1)^{-|X|-1}\partial j(\partial) (e_{\lambda}) \\
&=& \sum _{\nu, \mu} (-1)^{|e_{\nu}||X|} e_{\nu}  \bigg[
\frac{d b_{\nu \mu}}{dX}b_{\mu \lambda}+ (-1)^{(|e_{\mu}|-|e_{\nu}|-1)|X|} b_{\nu \mu} \frac{d b_{\mu \lambda}}{dX}\\
&+& (-1)^{|e_{\nu}|-|X|}d^B\left(\frac{d b_{\nu \lambda}}{dX}\right)\bigg] \\
&=& 0
\end{eqnarray*}
as desired.
\end{proof}

The next result proves a version of the Leibniz rule for the $j$-operators.

\begin{thm}\label{LR for D}
For all $\alpha, \beta \in \D$,  the $j$-operator  $j _X^{\mathcal{B}}\colon \D\to \D$  satisfies the equality:
\begin{equation}\label{LR for alpha}
j(\alpha \beta) = j(\alpha) \beta + (-1)^{|X||\alpha|} \alpha j(\beta).
\end{equation}
\end{thm}

\begin{proof}
Write $j =  j _X^{\mathcal{B}}$, and $\D = \E + \partial \circ \E= \E + \E \circ \partial$ for some $\partial \in \Diff _B(N)$.
To see the equality~\eqref{LR for alpha}, we prove a series of claims as follows.

Claim 1: for all $f, g \in \E$, the following equality holds:
\begin{equation}\label{LR1}
j ( f \partial g) =  j(f) \partial g +(-1)^{|f||X|}f j(\partial) g + (-1)^{(|f|+1)|X|}f\partial j(g).
\end{equation}
To see this, since $\partial g \in \D= \E + \E \circ \partial$, we write $\partial g = g_0 + g_1 \partial$ for some $g_0, g_1 \in \E$.
It follows from Corollary~\ref{cor20200316a} that
\begin{eqnarray*}
j( f \partial g) &=& j(fg_0) + j(fg_1 \partial) \\
&=& j(f)g_0+(-1)^{|f||X|}f j(g_0) + j(fg_1) \partial + (-1)^{(|f|+|g_1|)|X|}fg_1j(\partial) \\
&=& j(f)(g_0+g_1 \partial) + (-1)^{|f||X|}f \left(j(g_0) +j(g_1)\partial +(-1)^{|g_1||X|}g_1 j(\partial) \right) \\
&=& j(f) \partial g + (-1)^{|f||X|} f j (\partial g)\\
&=& j(f) \partial g + (-1)^{|f||X|}f \left( j(\partial)g + (-1)^{|X|} \partial j(g)\right).
\end{eqnarray*}

Claim 2: the following equality holds:
\begin{equation}\label{LR2}
j (\partial ^2 ) = j(\partial) \partial + (-1)^{|X|}  \partial j( \partial)
\end{equation}
To see this, note that if we set $h := \partial - \partial^{\mathcal{B}}$, then as we mentioned in~\ref{para20200315g}, $h$  belongs to $\E$. We also have
$\partial ^2 = h \partial ^{\mathcal{B}} + \partial ^{\mathcal{B}}h + h^2$  because $(\partial^{\mathcal{B}})^2=0$.
Then using Corollary~\ref{cor20200316a} and Lemma \ref{j-operator properties}(c) we have
\begin{eqnarray*}
j( \partial ^2 ) &=& j(h) \partial ^{\mathcal{B}} +(-1)^{|X|}  h j( \partial ^{\mathcal{B}}) +j( \partial ^{\mathcal{B}}) h +(-1)^{|X|}  \partial ^{\mathcal{B}} j(h) + j(h)h +(-1)^{|X|} hj(h) \\
&=& ( j(h) + j(\partial ^{\mathcal{B}} )) (h + \partial ^{\mathcal{B}}) +(-1) ^{|X|} (h + \partial ^{\mathcal{B}})( j(h) + j( \partial ^{\mathcal{B}})) \\
&=&  j(\partial) \partial + (-1)^{|X|}  \partial j( \partial).
\end{eqnarray*}

Claim 3: the following equality holds for all $f, g \in \E$:
\begin{align}\label{LR3}
j( f \partial g \partial ) = j(f) \partial g \partial +(-1)^{|f||X|} f j(\partial ) g \partial &+ (-1)^{(|f|+1) |X|}f \partial j(g) \partial \notag\\& + (-1)^{(|f|+1+|g|)|X|}f \partial g j(\partial).
\end{align}
To see this claim, since $g \partial\in \D$, write $g \partial = g_0 + \partial g_1$ for some $g_1, g_1 \in \E$. By the equalities~\eqref{LR1} and~\eqref{LR2} we get:
\begin{eqnarray*}
&&j( f \partial g \partial )\\ &=& j(f\partial g_0) + j( f \partial ^2 g_1) \\
 &=& j(f) \partial g_0 + (-1)^{|f||X|} f j(\partial) g_0 +(-1)^{(|f|+1)|X|}f \partial j(g_0) +  j(f) \partial ^2 g_1 \\ &+&
 (-1) ^{|f||X|} f j(\partial ^2) g_1 +  (-1)^{|f||X|} f\partial ^2 j(g_1)  \\
  &=&   j(f) \left(\partial g_0 + \partial ^2 g_1 \right) + (-1)^{|f||X|}f \left(j(\partial)g_0+(-1)^{|X|}\partial j(g_0) + j(\partial ^2) g_1 + \partial ^2 j(g_1)\right)  \\
    &=&  j(f) \partial g \partial + (-1)^{|f||X|}f\left( j(\partial)g_0+j(\partial)\partial g_1\right) \\
    &+&(-1)^{|f||X|}f\left[(-1)^{|X|}\partial j(g_0)+ (-1)^{|X|} \partial j(\partial) g_1 +  \partial ^2 j(g_1) \right]\\
    &=&  j(f) \partial g \partial +(-1)^{|f||X|} f j(\partial ) g \partial + (-1)^{(|f|+1) |X|}f \partial j(g \partial )
\end{eqnarray*}
where, in the second equality  we use the fact that $\partial ^2 \in \E$, as we discussed in \ref{yuji20201115a}.

Now consider the representations $\alpha=f_1+g_1\partial\in \D$ and $\beta=f_2+g_2\partial\in \D$, for some $f_1, f_2, g_1, g_2\in \E$. Then,~\eqref{LR for alpha} follows from~\eqref{LR1} and~\eqref{LR3}.
\end{proof}

We conclude this section by the following remark on $j$-operators.

\begin{para}\label{para20201112a}
We do not know how to generalize the notion of $j$-operator $j_X$ for one variable to the case where we have more than one variable in a way that good properties of the $j$-operators are preserved. More precisely, for a basis  $\mathcal{B}$ of $N$ and a set of variables $X_1,X_2,\ldots$, if we define the $j$-operator $j_{X_i}^\mathcal{B}$ with respect to a variable $X_i$ to be a map from $\D$ to itself that is described as in Definition~\ref{def j-op} using  the partial derivative operator $\partial/\partial X_i$ instead of  $d/dX$, then the derivation property from Lemma~\ref{LR for E} is not preserved. For example, consider the DG algebra extension $R\langle X_1, X_2, Y\mid dX_1=ab, dX_2=ac, dY=cX_1-bX_2\rangle$ of $R$, where $|X_1|=1=|X_2|$, $|Y|=2$, and $a,b,c\in R$ with $a$ non-unit. Then $j_{X_1}(dY)=c$, which is not a boundary in $B$. Hence, it follows from~\cite[Lemma 1.3.2]{GL} that $j_{X_1}$ is not a derivation.
\end{para}

\section{Characterization of $j$-operators}\label{sec20200329a}

This section is devoted to a characterization of $j$-operators in the set of weak $j$-operators; see Definition~\ref{def of Der} and Theorem~\ref{characterization} below. We will use this characterization in our (weak) lifting results later in the paper.

In this section we use the notation from Sections~\ref{sec20200314b} and~\ref{sec20200314c}.

\begin{defn}\label{def of Der}
Let $\Der _A(\D)$  be the set of  finite sums of graded $R$-linear mappings $\Delta\colon \D \to \D$ that satisfy the following properties:
\begin{enumerate}[(i)]
\item
$\Delta (a) = 0$  for  all $a \in A$;
\item
$\Delta (\E \cup \Diff_B(N)) \subseteq  \E$;
\item
$\Delta (\alpha\beta) = \Delta (\alpha) \beta + (-1) ^{|\Delta||\alpha|} \alpha\Delta(\beta)$
for all  $\alpha, \beta \in \D$.
\end{enumerate}
We call elements of   $\Der  _A(\D)$  \emph{weak $j$-operators}. Note that every weak $j$-operator $\Delta \in \Der _A(\D)$ is $A$-linear as well.
\end{defn}

Next results contain a list of properties that are satisfied by elements in $\Der _A(\D)$.

\begin{lem}\label{lem20200318s}
\begin{enumerate}[\rm(a)]
\item
All $j$-operators belong to  $\Der _A(\D)$.
In fact,  $j _X^{\mathcal B}$  is a homogeneous element of $\Der _A (\D)$ of degree $|j_X^{\mathcal B}| = -|X|$.
\item
If  $f \in \E$, then  $ad (f)$ is an element of  $\Der _A(\D)$.
If  $f$ is homogeneous, then  $ad (f)$ is also homogeneous with  $| ad (f) | = |f|$.
\item
If $\Delta, \Delta' \in \Der_A(\D)$, then $\Delta \pm \Delta', [\Delta, \Delta '] \in \Der _A(\D)$.
\item
If $\Delta \in \Der_A(\D)$ and   $|\Delta|$ is odd, then $\Delta ^2 \in \Der _A(\D)$.
\end{enumerate}
\end{lem}

\begin{proof}
(a) follows from Lemma~\ref{LR for E}, Corollary~\ref{cor20200316a}, and Lemma~\ref{j-operator properties}.

(b) follows from~\ref{para20200314b}, \ref{Jacobi}, and~\ref{prop20200315a}.

(c) and (d) are straightforward.
\end{proof}


\begin{lem}\label{useful lemma}
If $\Delta \in \Der _A(\D)$ and  $f \in \E$ are homogeneous, then $[\Delta, ad (f)] = ad (\Delta(f))$.
\end{lem}

\begin{proof}
For all $\alpha\in \D$ we have
\begin{eqnarray*}
[\Delta, ad (f)] (\alpha)\!\!\!\!\! &=&\!\!\!\!\! \Delta [f, \alpha] - (-1) ^{|\Delta ||f|}[f, \Delta (\alpha)] \\
&=&\!\!\!\!\!\Delta\left(f\alpha - (-1)^{|f||\alpha|} \alpha f\right)\! -\! (-1) ^{|\Delta ||f|}\!\left(f\Delta (\alpha) - (-1)^{|f|(|\Delta| + |\alpha|)} \Delta (\alpha)f\right)  \\
&=&\!\!\!\!\! \Delta (f)\alpha - (-1)^{|\alpha|(|\Delta| + |f|)} \alpha \Delta(f) \\
&=&\!\!\!\! ad (\Delta(f) )(\alpha)
\end{eqnarray*}
as desired.
\end{proof}

The following notation will be used later in the paper.

\begin{para}\label{notn20200320a}
For a basis  $\mathcal B = \{ e_{\lambda}\}$  of  $N$ as a free $B$-module, the set of \emph{orthogonal idempotents corresponding to $\mathcal B$} is a set $\{ \epsilon _{\lambda} \} \subseteq \E$, where  $\epsilon _{\lambda}$  are the $B$-linear transformations defined by the equalities
$$
\epsilon _{\lambda} (e_{\mu}) = \begin{cases} e _{\lambda} & (\lambda = \mu) \\ 0 & (\lambda \not= \mu).  \end{cases}
$$
Note that   $\epsilon _{\lambda} ^2 = \epsilon _{\lambda}$ and $\epsilon _{\lambda}\epsilon _{\mu}=0$ if $\lambda \not= \mu$.
\end{para}

\begin{prop}\label{kernel of ad}
Let  $f \in \E$, and assume that  $ad (f) =0$.
Then $f$ is a cycle in $B$ meaning that there is an element  $b \in B$ satisfying  $d^B (b) = 0$ and $f = \ell _b$.
\end{prop}

\begin{proof}
Let  $\mathcal B = \{ e_{\lambda}\}$  be a basis of  $N$ as a free $B$-module and $\{ \epsilon _{\lambda} \} \subseteq \E$ be the set of orthogonal idempotents corresponding to $\mathcal B$. Since  $ad (f) =0$ and each $\epsilon _{\lambda}$ is a graded map with $|\epsilon _{\lambda}|=0$, we have  $f  \epsilon_{\lambda} = \epsilon_{\lambda} f$  for each $\lambda$.
Hence,  $\epsilon_{\lambda}f \epsilon_{\mu} = f \epsilon_{\lambda}\epsilon_{\mu} = 0$ if $\lambda \not= \mu$.
It follows that  $f$   is a diagonal matrix with this basis, that is,  $f = \sum _{\lambda} \epsilon_{\lambda}b_{\lambda}$  for  some  $b_{\lambda} \in B$.
We define  $\epsilon _{\lambda \mu} \in \E$ by
$$
\epsilon _{\lambda \mu } (e_{\nu}) = \begin{cases} e _{\lambda} & (\nu = \mu) \\ 0 & (\nu \not= \mu).  \end{cases}
$$
So, $\epsilon _{\lambda} = \epsilon _{\lambda \lambda}$ and  $|\epsilon _{\lambda \mu }| = |e_{\lambda}|-|e_{\mu}|$.
We now have
\begin{eqnarray*}
0 &=& ad (f) ( \epsilon _{\lambda \mu })\\
&=& f  \epsilon _{\lambda \mu } -(-1)^{|f|(|e_{\lambda}|-|e_{\mu}|)} \epsilon _{\lambda \mu }f  \\
&=& (-1)^{|b_{\lambda}|(|e_{\lambda}|-|e_{\mu}|)} \epsilon _{\lambda \mu } b _{\lambda } -
(-1)^{|f|(|e_{\lambda}|-|e_{\mu}|)} \epsilon _{\lambda \mu } b_{\mu}.
\\
\end{eqnarray*}
Note that $|f|=|b_{\lambda}|$ for all $\lambda$. Therefore,  $b_{\lambda} = b_{\mu}$  for all $\lambda, \mu$ and hence,  $f$ is a scalar matrix, that is,
$f = \ell_b$  for an element $b \in B$.
Moreover,  the fact that $ad (f) (\partial)=0$ for all $\partial\in \Diff_B(N)$ implies that $\ell _{d^Bb} =[\ell _b, \partial] = 0$. Hence,  $d^Bb=0$.
\end{proof}

\begin{cor}\label{cor to kernel of ad}
Let  $f \in \E$, and assume that  $ad (f) =0$ as an element of  $\Der_A(\D)$.
If $|f| < 0$, then $f=0$.
\end{cor}

\begin{proof}
By virtue of Proposition~\ref{kernel of ad} we have  $f = \ell _{b}$ for some $b \in B$.
Since $|b| = |f| < 0$ and $B$  is non-negatively graded, we must have $b=0$.
\end{proof}

Next, we give a characterization of $j$-operators in $\Der _A (\D)$. The following is the main result of this section.

\begin{thm}\label{characterization}
Let  $\Delta \in \Der _A(\D)$ be homogeneous and assume that the following conditions hold:
\begin{enumerate}[\rm(1)]
\item  $\Delta (X) = 1$ if $|X|$ is odd, and  $\Delta (X^{(n)}) = X^{(n-1)}$ for $n >0$  if  $|X|$ is even.
\item  There is a set of orthogonal idempotents  $\{ \epsilon _{\lambda} \}\subseteq \E$  corresponding to a basis  $\mathcal B=\{ e_{\lambda}\}$ of $N$ as a free $B$-module such that  $\Delta (\epsilon _{\lambda})=0$  for all $\lambda$.
\end{enumerate}
Then  we have the equality  $\Delta = j _X ^{\mathcal B}$.
\end{thm}

Note that the $j$-operator $j _X ^{\mathcal B}$ satisfies conditions (1) and (2).
Hence, (1) and (2) together create a necessary and sufficient set of conditions for $\Delta$ to be a $j$-operator.

\begin{proof}
By (1) we have  $|\Delta | = -|X|$.
Define  $\{\epsilon _{\lambda \mu}\}\subseteq \E$ as in the proof of Proposition~\ref{kernel of ad} and note that $\epsilon _{\lambda} = \epsilon _{\lambda \lambda}$.

Claim 1: $\Delta (\epsilon _{\lambda \mu} ) =0$  for all $\lambda, \mu$.

To prove this, recall from the definition of $\Delta$ that $\Delta (\epsilon _{\lambda \mu} )  \in \E$. We also have
$\epsilon _{i}  \Delta (\epsilon _{\lambda \mu} ) \epsilon _{j}
= \Delta ( \epsilon _{i}  \epsilon _{\lambda \mu} \epsilon _{j} ) =0$ if   $i \not= \lambda$ or $j \not= \mu$.
Hence, $\Delta (\epsilon _{\lambda \mu} ) = \epsilon _{\lambda \mu} b$  for some $b \in B$.
Notice that  $|b| = |\Delta | = -|X| < 0$, and since  $B$ is non-negatively graded, we have  $b =0$.

Claim 2: $\Delta (f)=j _X^{\mathcal B}(f) $  for all $f \in \E$.

For simplicity write $j=j _X^{\mathcal B}$. Note that both sides of the equality belong to $\E$. Hence, to prove the equality, it is enough to show
$\epsilon _{\lambda}\Delta (f) \epsilon _{\mu} =\epsilon _{\lambda} j(f) \epsilon _{\mu}$ for all $\lambda, \mu$, that is,
$\Delta ( \epsilon _{\lambda}f \epsilon _{\mu}) =j (\epsilon _{\lambda} f \epsilon _{\mu})$ for all $\lambda, \mu$.
So, it is enough to prove the equality when  $f = \epsilon _{\lambda \mu} b$ for $b \in B$.
Since by Claim 1, $\Delta (\epsilon _{\lambda \mu} ) = j(\epsilon _{\lambda \mu} )=0$, it is enough to show
$\Delta (b) = j(b)$  for $b \in B$.

In case where $|X|$ is odd, let $b =a_0 + Xa_1$ with $a_i \in A$. Since  $\Delta (a_i)=0$ and $\Delta (X)=1$, we have $\Delta (b) = a_1 = j (b)$.

In case where $|X|$ is even, let $b = \sum_{i=0}^n X^{(i)} a_i\in B$ for $a_i\in A$. Since  $\Delta (X^{(n)}) = X^{(n-1)}$ we have
$\Delta (b) = \sum_{i=1}^n X^{(i-1)} a_i = j(b)$.

Claim 3: $\Delta (\partial ^{\mathcal B}) =0\  ( = j( \partial ^{\mathcal B}))$.

Set $\partial = \partial ^{\mathcal B}$ and note that if $\lambda \not= \mu$, then
$$
(\epsilon _{\lambda} \partial  \epsilon _{\mu}) \left(\sum _\nu e_{\nu} b_{\nu}\right)  = \epsilon _{\lambda}\partial (e_{\mu} b_{\mu}) = \epsilon _{\lambda} (e_{\mu} d^Bb_{\mu} )=\epsilon _{\lambda} (e_{\mu}) d^Bb_{\mu}=0.
$$
Hence,  $\epsilon _{\lambda} \Delta( \partial)  \epsilon _{\mu}=\Delta ( \epsilon _{\lambda} \partial  \epsilon _{\mu})=0$ if $\lambda \not= \mu$.
Since $ \Delta( \partial)  \in \E$, we have  $ \Delta( \partial) = \sum _{\lambda} \epsilon _{\lambda}b_{\lambda}$, which means that $ \Delta( \partial)$ is a diagonal matrix.
Since  $|\Delta (\partial) | = |\Delta | -1 = -|X|-1$ is negative and $| \epsilon _{\lambda}|=0$, it follows that $|b _{\lambda}|< 0$ for all $\lambda$. Thus, for all $\lambda$ we have $b_{\lambda} =0$ and therefore, $\Delta( \partial) =0$.

Since by~\ref{prop20200315a} we have $\D = \E + \E \circ \partial$,
it follows from Claims 2 and 3 and Corollary~\ref{cor20200316a} that   $\Delta  = j$  as mappings  from $\D$ to itself.
\end{proof}

Theorem~\ref{characterization} enables us to find the conditions that are needed for an element $\Delta \in \Der _A(\D)$ to be a $j$-operator in the specific cases of $|X|$ being even or odd. First we treat the case where $|X|$ is even.

\begin{prop}\label{even}
Suppose that $|X|$ is even. Let  $\Delta \in \Der _A(\D)$ be homogeneous, and assume that the following conditions hold:
\begin{enumerate}[\rm(1)]
\item  $N$ is bounded below.
\item  $\Delta (X^{(n)}) = X^{(n-1)}$ for $n >0$.
\end{enumerate}
Then $\Delta |_{\E}\colon \E \to \E$ is surjective and there is a basis $\mathcal B= \{ e_{\lambda}\}$ of  $N$ as a graded $B$-module such that  $\Delta = j_X ^{\mathcal B}$.
\end{prop}

\begin{proof}
For an arbitrary element $f \in \E$ let $$f_+ = X \Delta (f) - X ^{(2)} \Delta ^2 (f) + \cdots + (-1)^{n+1} X^{(n)}\Delta ^n (f) + \cdots.$$
Since  $N$ is bounded below, $f_+$ is well-defined as a mapping $N \to N$.
One can verify that $\Delta (f _+) = \Delta (f)$ and setting   $f_0 = f -f_+$, we have  $\Delta (f_0) =0$. Set $$F = X f_0 + X^{(2)}  \Delta (f) - 2 X^{(3)} \Delta ^2(f) + \cdots + (-1)^{n+1} n X^{(n+1)} \Delta ^n (f) + \cdots.$$
Note that $F \in \E$ as well and  $\Delta (F) =f$. Hence, $\Delta |_{\E}\colon \E \to \E$ is a surjective map.

Now we prove that $\Delta$ is equal to the $j$-operator with respect to a basis for $N$.
By Theorem~\ref{characterization} it suffices to show that there is a basis $\mathcal B$ and a set of orthogonal idempotents $\{ \epsilon _{\lambda}\}\subseteq \E$ corresponding to $\mathcal B$  such that  $\Delta ( \epsilon _{\lambda}) =0$ for all $\lambda$.

Set $\E _0 := \Ker \Delta \cap \E = \{ f \in \E \ | \ \Delta (f) =0 \}$ and $\E_+ := \sum _{i >0} X^{(i)} \E$.
By the above discussion we have $\E = \E_0 + \E_+$.
Note that  $\E_0 \subseteq \E$  is a subring, while  $\E_+ \subseteq \E$  is an ideal.
We claim the following
\begin{equation}\label{eq20200318a}
\E = \E_0 \oplus \E_+\quad\quad\quad \text{and} \quad\quad\quad \E_0\circ \E_+ = \E_+\circ \E_0 = \E_+ \supseteq  \E_+\circ \E_+.
\end{equation}
To prove this, it is enough to show $\E_0 \cap  \E_+ = \{ 0\}$. (The rest is trivial.)
Assume that $f \in \E_0 \cap  \E_+$ and write
$f= \Sigma_{i\geq 1}X^{(i)}f_i$. The fact that  $\Delta (f)=0$ implies that  $f_1+\Sigma_{i\geq 1} X^{(i)}(\Delta(f_i) +f_{i+1})=0$. Hence,
$f_1 \in  \sum _{i \geq 1} X^{(i)} \E$.
This shows that $f \in  \sum _{i \geq 2} X^{(i)} \E$.
By induction on $n$, we see that  $f \in  \sum _{i \geq n} X^{(i)} \E$.
Since $N$ is bounded below, it follows that $f(x) =0$  for all $x \in N$.
This finishes the proof of the equality $\E = \E_0 \oplus \E_+$.

Now let  $\{ e_{\lambda}\}$ be a basis of $N$ as a free $B$-module, and let  $\{ \epsilon _{\lambda}\}\subseteq \E$  be the corresponding orthogonal idempotents.
By~\eqref{eq20200318a} an element $\epsilon _{\lambda} \in \E$ decomposes as
$\epsilon _{\lambda} = (\epsilon _{\lambda} )_0 + (\epsilon _{\lambda} )_+ \in \E_0 \oplus \E_+$.
Since  $\epsilon _{\lambda} ^2=\epsilon _{\lambda}$, by~\eqref{eq20200318a} we have
$(\epsilon _{\lambda} )_0 ^2=(\epsilon _{\lambda} )_0$, and since $(\epsilon _{\lambda} )(\epsilon _{\mu} )=0$ for $\lambda \not= \mu$, we obtain the equality $(\epsilon _{\lambda} )_0 (\epsilon _{\mu} )_0 =0$.
Therefore, $\{ (\epsilon_{\lambda})_0 \}$  is a set of orthogonal idempotents.

Now set  $(e_{\lambda})_0 :=  (\epsilon_{\lambda})_0(e_{\lambda})$.
Since  $ (\epsilon_{\lambda})_0 \equiv  \epsilon_{\lambda} \pmod {\E_+}$, we see that
$(e_{\lambda})_0 \equiv e_{\lambda} \pmod {\sum _{i >0} X^{(i)} N}$.
Hence, similar to Nakayama's Lemma, we see that $\{ (e_{\lambda})_0\}$  is a generating set of $N$.
It is straightforward to check that $\{ (e_{\lambda})_0\}$ is  linearly independent over  $B$, so it is a basis of $N$.
Note that $\Delta((\epsilon _{\lambda} )_0)=0$,
$(\epsilon _{\lambda} )_0((e_{\lambda})_0 )= (e_{\lambda})_0$, and $(\epsilon _{\lambda} )_0((e_{\mu})_0 )= 0$ if $\lambda\not= \mu$.
Therefore, $\mathcal B=\{ (e_{\lambda})_0\}$ is the basis for $N$ as a free $B$-module with the orthogonal idempotents $\{ (\epsilon _{\lambda})_{0}\}\subseteq \E$ corresponding to $\mathcal B$ with the desired property and this finishes the proof.
\end{proof}

Next, we give an odd version of Proposition~\ref{even}.

\begin{prop}\label{odd}
Suppose that $|X|$ is odd.
Let $\Delta \in \Der _A(\D)$ be homogeneous, and assume that the following conditions hold:
\begin{enumerate}[\rm(1)]
\item  $\Delta (X) = 1$.
\item $\Delta ^2 =0$
\end{enumerate}
Then there is a basis $\mathcal B= \{ e_{\lambda}\}$ of  $N$ as a free $B$-module such that  $\Delta = j_X ^{\mathcal B}$.
\end{prop}

\begin{proof}
By virtue of Theorem \ref{characterization}, it suffices to show that there is a basis $\{ e_{\lambda}\}$ of $N$ such that  $\Delta (\epsilon _{\lambda})=0$ for its associated orthogonal idempotents $\epsilon_{\lambda}$.
Note that for  $f \in \E$  we have the equality
$f = \Delta (Xf) + X\Delta (f)$.
Setting   $\E_0 = \Ker \Delta \cap \E$ and  $\E_1 = X \E_0$ we see that
$\Delta (Xf) \in \E_0$ and  $X\Delta (f) \in \E_1$.
It is straightforward to check that  $\E_0 \cap \E_1 =\{ 0\}$.
Hence,
$\E = \E_0 \oplus \E_1$, $\E_0\circ \E_0 \subseteq \E_0$, $\E_0\circ \E_1 \subseteq \E_1$, and $\E_1\circ \E_1 =\{ 0\}$.
(Note that $\E$ is a $\mathbb{Z}/2\mathbb{Z}$-graded ring.)
Let  $\mathcal B=\{ e_{\lambda}\}$  be any basis of $N$ as a free $B$-module, and let  $\{ \epsilon _{\lambda}\}\subseteq \E$ be the corresponding orthogonal idempotents.
Taking the degree $0$ part of $\epsilon _{\lambda}$,
we have a set of orthogonal idempotents  $\{ (\epsilon _{\lambda})_0 \}\subseteq \E$ corresponding to $\mathcal B$ that satisfies the desired property.
\end{proof}

The following is another application of Theorem~\ref{characterization}.

\begin{prop}\label{automorphism}
Let  $u \in \E$ be a unit, i.e., $u\colon N \to N$ is an automorphism, and consider the graded bases  $\mathcal B = \{ e_{\lambda} \}$ and $\mathcal B'=\{e'_{\lambda}\}$ of $N$ as a free $B$-module, where $e'_{\lambda} = u (e_{\lambda})$.
Then, we have the equality $j_X^{\mathcal B} - j_X^{\mathcal B'} = ad (\alpha )$, where $\alpha = j_X^{\mathcal B} (u) u ^{-1} \in \E$.


Moreover, $j_X^{\mathcal B} = j_X^{\mathcal B'}$  if and only if
there is an isomorphism $v\colon  \sum_{\lambda} e_{\lambda}A \to \sum_{\lambda} e'_{\lambda} A$ of graded $A$-modules such that $u = v \otimes _A B$.
\end{prop}

\begin{proof}
For $f \in \D$,  if we set $f( e_{\lambda} ) = \sum _{\mu} e_{\mu} F _{\mu \lambda}$ and $f( e'_{\lambda} ) = \sum _{\mu} e'_{\mu} F' _{\mu \lambda}$,
where  $F _{\mu \lambda}, F' _{\mu \lambda} \in B$,  then the matrix representation of  $u ^{-1} f u$ with respect to the basis $\mathcal B$  is $[F' _{\mu \lambda}]$.
Hence, we have $j_X^\mathcal B ( u ^{-1} f u ) = u ^{-1} j_X^{\mathcal B'} (f) u$.
Therefore,
$$
u ^{-1}j_X ^{\mathcal B'} (f) u = j_X^\mathcal B ( u ^{-1})f u + u ^{-1} j_X^\mathcal B ( f ) u  + (-1) ^{|X||f|} u ^{-1} f j_X^\mathcal B ( u ).
$$
Since  $ j_X^\mathcal B ( u )= \alpha u$, from $u u^{-1} =1$ we see that  $ j_X^\mathcal B ( u ^{-1}) =-u ^{-1} \alpha$.
Thus,
$$
u ^{-1}j_X ^{\mathcal B'} (f) u = u ^{-1} (-\alpha f  + j_X^\mathcal B ( f )  + (-1) ^{|X||f|} f \alpha)  u.
$$
Hence, we have
$$
j_X ^{\mathcal B'} (f) =-\alpha f  + j_X^\mathcal B ( f )  + (-1) ^{|X||f|} f \alpha.
$$
Now the equality $j_X^{\mathcal B} - j_X^{\mathcal B'} = ad (\alpha )$ follows from the fact that  $|\alpha | = -|X|$.

Finally, note that if  $j_X^{\mathcal B} = j_X^{\mathcal B'}$, then $\alpha =0$ by Corollary~\ref{cor to kernel of ad}. This means that  $j_X^{\mathcal B} (u) =0$ and hence, $u$ is a matrix whose components are in $A$.
\end{proof}

\section{Weak liftability of DG modules}\label{sec20200314d}

As consequences of Theorem~\ref{characterization}, in this section, we record some lifting and weak lifting results for DG modules along simple extensions of DG algebras. For more results on the notions of liftings and weak liftings and an extensive background on these subjects see also~\cite{auslander:lawlom, nasseh:lql, nassehyoshino, OY, yoshino}.

The notation used in this section comes from the previous sections.

\begin{defn}\label{weak lifting defn}
A right DG $B$-module $N$ is \emph{liftable} to $A$ if there is a right DG $A$-module $M$ such that $N \cong M\lotimes_A B$ (or $N\cong M\otimes_A B$ if $M$ and $N$ are semifree)
in the derived category $\D(B)$. In such a case we say that $M$ is a lifting of $N$ to $A$. A DG $B$-module $L$ is \emph{weakly liftable} to $A$ if there are positive integers $a_1,\ldots,a_r$ such that the finite direct sum $N \oplus N(-a_1) \oplus \cdots \oplus N(-a_r)$ is liftable to $A$.
\end{defn}

\begin{lem}\label{lifting}
Let  $(N, \partial)$ be a DG $B$-module and $\mathcal B = \{ e_{\lambda}\}$  be a basis of $N$ as a free $B$-module.
\begin{enumerate}[\rm(a)]
\item
If  $j_X^{\mathcal B}(\partial)=0$, then $(N, \partial)$ has a lifting  $M=\sum _{\lambda} e_{\lambda} A$ to  $A$.
\item
Conversely, if  $(N, \partial)$  has a lifting $M=\sum _{\lambda} e_{\lambda} A$ to $A$, then there is a graded $B$-module automorphism $u\in \E$  such that $j_X^{\mathcal B'}(\partial)=0$, where $\mathcal B' = \{ u(e_{\lambda})\}$.
\end{enumerate}
\end{lem}

\begin{proof}
(a) If $j_X^{\mathcal B}(\partial)=0$,  then  $\partial (M) \subseteq M$. Thus, $(M, \partial |_{M})$ is a lifting of the DG $B$-module $(N,\partial)$ to $A$.

(b)
If  $(M,\partial^M)$ is a lifting of $N$ to $A$, then we have an isomorphism $(N, \partial ) \cong (M , \partial^M) \otimes _A B$ of DG $B$-modules, that is, there is a graded $B$-module automorphism $u\in \E$ satisfying $\partial  = u ( \partial ^M \otimes _A B) u ^{-1}$.
Then, as in the first paragraph of the proof of Proposition~\ref{automorphism}, we have
$$
j_X^{\mathcal B'} (\partial ) =
j_X^{\mathcal B'} (u ( \partial ^M \otimes _A B) u ^{-1}) =
uj_X^{\mathcal B} (\partial ^M \otimes _A B)u^{-1} = 0
$$
as desired.
\end{proof}

The next result follows from Theorem~\ref{characterization} and Lemma~\ref{lifting}.

\begin{cor}\label{liftable condition}
Let  $(N, \partial)$ be a DG $B$-module and $\Delta \in \Der _A( \D)$ be homogeneous.
Assume that the following conditions hold:
\begin{enumerate}[\rm(1)]
\item  $\Delta (X) = 1$ if $|X|$ is odd, and  $\Delta (X^{(n)}) = X^{(n-1)}$ for $n >0$  if  $|X|$ is even.
\item  There is a set of orthogonal idempotents  $\{ \epsilon _{\lambda} \}\subseteq \E$  corresponding to a basis  $\mathcal B=\{ e_{\lambda}\}$ of $N$ as a free $B$-module such that  $\Delta (\epsilon _{\lambda})=0$  for all $\lambda$.
\item $\Delta (\partial ) =0$.
\end{enumerate}
Then the DG $B$-module $(N, \partial)$ has a lifting  $(M, \partial |_{M})$ to  $A$,
where  $M = \sum _{\lambda} e_{\lambda} A$.
\end{cor}

\begin{para}
In Corollary \ref{liftable condition}, if $(N, \partial)$ is a semifree DG $B$-module with the semibasis  $\mathcal B$, then $(M, \partial |_M)$ is a semifree DG $A$-module as well and  $\mathcal B$  is its semibasis.
\end{para}




The following is an immediate consequence of Proposition~\ref{even} and Lemma~\ref{lifting}.

\begin{prop}\label{cor20200318f}
Assume that $|X|$ is even.
Let  $(N, \partial)$ be a DG $B$-module and $\Delta \in \Der _A(\D)$ be homogeneous.
Assume that the following conditions hold:
\begin{enumerate}[\rm(1)]
\item  $N$ is bounded below.
\item  $\Delta (X^{(n)}) = X^{(n-1)}$ for $n >0$.
\item $\Delta (\partial ) =0$.
\end{enumerate}
Then the DG $B$-module $(N, \partial)$ is liftable to $A$.
\end{prop}


This enables us to prove the following result~\cite[Theorem(1), p. 343]{OY}.

\begin{cor}\label{cor20201111s}
Assume that $|X|$ is even.
Let  $L$ be a DG $B$-module that is bounded below.
If  $\Ext _B^{|X|+1}(L, L)=0$, then  $L$ is liftable to $A$.
\end{cor}

\begin{proof}
Without loss of generality we can replace $L$ with a semifree resolution $(N,\partial)$ and prove the assertion for $N$. Note that $N$ is free as an underlying graded $B$-module that is bounded below and we have $\Ext _B^{|X|+1}(N, N)=0$.
Then, by Lemma~\ref{j-operator properties}(c) we see that  $j_X (\partial)$ is a cycle in $\Hom _B (N,N(-|X|-1))$. By our Ext vanishing assumption, there is $\gamma \in \E$ with $|\gamma|=-|X|$ such that  $j_X(\partial) = [  \partial, \gamma]$.
Setting  $\Delta = j_X + ad (\gamma)$, by Lemma~\ref{lem20200318s} we have $\Delta \in \Der _A(\D)$, $\Delta (X^{(n)}) =j_X (X^{(n)})= X^{(n-1)}$, and $\Delta (\partial)=0$. Now the assertion follows from Proposition~\ref{cor20200318f}.
\end{proof}

The following is a consequence of Proposition~\ref{odd} and Lemma~\ref{lifting}.

\begin{cor}\label{odd lift}
Assume that $|X|$ is odd.
Let  $(N, \partial)$ be a DG $B$-module  and $\Delta \in \Der _A(\D)$ be homogeneous.
Assume that the following conditions hold:
\begin{enumerate}[\rm(1)]
\item  $\Delta (X) = 1$.
\item $\Delta (\partial ) =0$.
\item $\Delta ^2 =0$.
\end{enumerate}
Then the DG $B$-module $(N, \partial)$ is liftable to $A$.
\end{cor}

\begin{para}
Comparing to the case where $|X|$ is even, we require the condition $\Delta ^2=0$ for $\Delta$  to be a $j$-operator in Proposition~\ref{odd}.
Without this condition we only get weakly liftability that is not necessarily a liftable condition; see the next result.
\end{para}

\begin{prop}\label{odd weak}
Suppose that $|X|$ is odd.
Let  $(N, \partial)$ be a DG $B$-module  and  $\Delta \in \Der _A(\D)$ be homogeneous.
Assume that the following conditions hold:
\begin{enumerate}[\rm(1)]
\item  $\Delta ^2= ad (\alpha )$  for some $\alpha \in \E$.
\item  $\Delta (X) = 1$.
\item $\Delta (\partial ) =0$.
\end{enumerate}
Then the DG $B$-module $(N, \partial)$ is weakly liftable to $A$.
More precisely, the two-fold extension  $N^{\sharp} = N \oplus N(-|X|)$ with differential $\partial^\sharp =
\left[\begin{smallmatrix}\partial & 0 \\ 0 & -\partial  \end{smallmatrix}\right]$ is liftable to $A$.
\end{prop}

We need the following lemma to prove Proposition~\ref{odd weak}.

\begin{lem}\label{sublemma}
Consider  the assumptions of Proposition \ref{odd weak}. If there is  $\beta \in \E$ with
\begin{equation}\label{beta}
\alpha = -\beta ^2, \quad \Delta (\beta )=0, \ \ \text{and} \ \ \ [\partial, \beta] = 0
\end{equation}
then $(N, \partial)$ is liftable to $A$.
\end{lem}

\begin{proof}
Note that  $|\Delta|=-|X|$ is odd  and $|\alpha | = 2 |\Delta | = - 2|X|$ is even  in the setting of Proposition~\ref{odd weak}.
Setting  $\Gamma = \Delta + ad (\beta) \in \Der _A(\D)$ we see that
$\Gamma (X) = \Delta (X) =1$  and  $\Gamma (\partial ) = \Delta (\partial ) - [\beta, \partial]=0$.
Moreover, noting  that  $|\beta | =  |\Delta | $ is odd we have
$$
\Gamma ^2 = \Delta ^2  + \Delta ad (\beta) + ad(\beta) \Delta +\left(  ad(\beta) \right) ^2
= ad (\alpha) + [\Delta, ad (\beta)] + ad (\beta ^2)
$$
which is zero by (\ref{beta}) and Lemma~\ref{useful lemma}.
Hence, it follows from Corollary \ref{odd lift} that  $(N, \partial)$ is liftable to $A$.
\end{proof}

\begin{para}\label{para20200427a}
Let $(N,\partial^N)$ be a DG $B$-module. The two-fold extension of $N$ of degree $k$  is
$N^{\sharp} = N \oplus N(k)$ with the differential $\partial^\sharp =  \left[\begin{smallmatrix}\partial^N & 0 \\ 0 & (-1)^{k} \partial^N  \end{smallmatrix}\right]$. Note that  $N$ is a right DG $B$-module, and so is $N^{\sharp}$. The right DG $B$-module structure on $N^{\sharp}$ comes from the formula $\left[\begin{smallmatrix} x\\y \end{smallmatrix}\right]\cdot b=\left[\begin{smallmatrix} xb\\yb \end{smallmatrix}\right]$ while its left DG $B$-module structure is given by
\begin{equation}\label{eq20200528a}
\ell_b(\left[\begin{matrix}x \\ y \end{matrix}\right])=\left[\begin{matrix}\ell_b(x) \\ (-1)^{|b|k}\ell_b(y) \end{matrix}\right]
\end{equation}
for all $b\in B$ and all $\left[\begin{smallmatrix}x \\ y \end{smallmatrix}\right]\in N^{\sharp}$.
In this case, $\E^\sharp = \End^* _B(N^\sharp) = \left[\begin{smallmatrix}\E & \E(-k) \\ \E(k)  & \E  \end{smallmatrix}\right]$ and $\D^\sharp =  \E^\sharp + \E^\sharp \circ\Diff _B(N^\sharp)$.
Note that  $\partial ^\sharp \in \Diff _B(N^\sharp)$, and hence, $\D^\sharp =  \E^\sharp + \E^\sharp \circ\partial ^\sharp$.

Every homogeneous element $\Delta \in \Der _A(\D)$ is extended to  $\Delta^\sharp \colon \D^\sharp \to \D^\sharp$  by defining
$$
\Delta ^\sharp
\left[\begin{matrix} a & b \\ c  & d  \end{matrix}\right]
=
\left[\begin{matrix} \Delta (a) & \Delta (b) \\ (-1)^{|\Delta |k} \Delta (c)  & (-1)^{|\Delta |k}\Delta (d)  \end{matrix}\right].
$$
In particular, we have
$$
\Delta ^\sharp (  \partial ^\sharp ) =
\left[\begin{matrix} \Delta (\partial^N) & 0 \\ 0  & (-1)^{k+|\Delta|k}\Delta (\partial^N )  \end{matrix}\right].
$$
Note that $\Delta ^\sharp$ satisfies the Leibniz rule.
In fact,
let
$\alpha = \left[\begin{smallmatrix}a & b  \\ c & d \end{smallmatrix}\right]$ and
$\alpha '= \left[\begin{smallmatrix}a' & b'  \\ c' & d' \end{smallmatrix}\right]$ be elements of $\D^\sharp$.
If  $n$  is the degree of $\alpha$ as an element of $\D^\sharp$, then
$|a| = |d| =n$, $|b| = n-k$,  and  $|c| = n+k$ as elements of  $\E$. Now one can check that the equality
$$
\Delta ^\sharp (\alpha \alpha')=\Delta ^\sharp (\alpha) \alpha' + (-1)^{|\Delta^\sharp| n } \alpha \Delta ^\sharp (\alpha ')
$$
holds.
Therefore,  $\Delta ^\sharp \in \Der _A(\D^\sharp)$.

Also, since each element  $X^{(m)}$ acts on $N ^\sharp$ diagonally as in~\eqref{eq20200528a}, we have
\begin{align*}
\Delta^\sharp(X^{(m)}) &= \Delta^\sharp \left[\begin{matrix} X^{(m)} &0\\0& (-1)^{|X^{(m)}|k}X^{(m)} \end{matrix}\right]\\
&= \left[\begin{matrix} \Delta(X^{(m)}) &0\\0& (-1)^{\left(|X^{(m)}|-|X|\right)k}\Delta(X^{(m)}) \end{matrix}\right].
\end{align*}
(Here, we assume that $X^{(0)}=1$, $X^{(1)}=X$, and  $X^{(m)}$=0 for $m\geq 2$ if $|X|$ is odd.)
\end{para}

\begin{para} \emph{Proof of Proposition~\ref{odd weak}}.
Note that  $\alpha \in \E$ is a cycle, i.e.,  $[\partial, \alpha]=0$.
In fact, $ad ( \alpha ) (\partial ) = \Delta^2(\partial ) = 0$ by conditions (1) and (3).
Also note that  $\Delta (\alpha) =0$.
In fact, by Lemma~\ref{useful lemma} we have $0 = [\Delta, \Delta ^2] = [\Delta, ad (\alpha)] = ad (\Delta (\alpha))$, where  $|\Delta (\alpha)| = |\Delta | + |\alpha | = 3 |\Delta | = -3|X|$ is negative. Hence,  $\Delta (\alpha) =0$  by Corollary \ref{cor to kernel of ad}.

Consider the two-fold graded $B$-free module  $N^\sharp= N \oplus N(-|X|)$.
By~\ref{para20200427a}, we have  $\partial ^\sharp=\left[\begin{smallmatrix}\partial & 0 \\ 0 & -\partial  \end{smallmatrix}\right] \in \Diff _B(N^\sharp)$ and $\D^\sharp =  \E^\sharp + \E^\sharp\circ \partial ^\sharp$.
Furthermore, $\Delta$ is extended to  $\Delta^\sharp \colon \D^\sharp \to \D^\sharp$  by $\Delta ^\sharp
\left[\begin{smallmatrix} a & b \\ c  & d  \end{smallmatrix}\right]
=
\left[\begin{smallmatrix} \Delta (a) & \Delta (b) \\ -\Delta (c)  & -\Delta (d)  \end{smallmatrix}\right]$
and  $\Delta ^\sharp \in \Der _A(\D^\sharp)$. We also have $\Delta^\sharp ( X) = \left[\begin{smallmatrix} \Delta (X) & 0 \\ 0  & \Delta (X)  \end{smallmatrix}\right] = 1$ and $\Delta^\sharp (\partial ^\sharp)=
\left[\begin{smallmatrix} \Delta (\partial) & 0 \\ 0  & \Delta (\partial )  \end{smallmatrix}\right]=0$, by our assumptions.
Setting $\alpha ^\sharp = \left[\begin{smallmatrix} \alpha & 0 \\ 0 & \alpha \end{smallmatrix}\right] \in \E^\sharp$, we note that $(\Delta ^\sharp)^2 = ad ( \alpha^\sharp )$ and $\Delta^\sharp (\alpha ^\sharp ) =0$.
Therefore, $\partial ^\sharp$, $\Delta ^\sharp$, and $\alpha ^\sharp$ satisfy  the same assumptions as $\partial$, $\Delta$, and $\alpha$ in Proposition~\ref{odd weak}.

Now set $\beta ^\sharp =  \left[\begin{smallmatrix} 0 & \alpha \\ -1 & 0 \end{smallmatrix}\right] \in \E^\sharp$.
Then we can verify that $\alpha ^\sharp = - (\beta^\sharp)^2$,
$\Delta ^\sharp(\beta ^\sharp)=\left[\begin{smallmatrix} 0 & \Delta (\alpha ) \\ 0 & 0 \end{smallmatrix}\right]=0$, and  $[\partial ^\sharp, \beta^\sharp] = \left[\begin{smallmatrix} 0 & [\partial, \alpha] \\0 & 0 \end{smallmatrix}\right] = 0$.
Thus, it follows from Lemma \ref{sublemma} that $(N^\sharp, \partial^\sharp)$  is liftable to $A$.
\qed
\end{para}

We can give a proof for~\cite[Theorem 3.6]{nassehyoshino} that uses the notion of $j$-operators.

\begin{cor}\label{cor20201111a}
Assume that $|X|$ is odd. Let  $L$ be a DG $B$-module such that
$\Ext _B^{|X|+1}(L, L)= 0$. Then  $L$ is weakly liftable to $A$.
\end{cor}

\begin{proof}
Without loss of generality we can replace $L$ with a semifree resolution $(N,\partial)$ and prove the assertion for $N$. Note that $N$ is free as an underlying graded $B$-module and we have $\Ext _B^{|X|+1}(N, N)=0$.
By Lemma~\ref{j-operator properties}(c) we see that  $j_X (\partial)$ is a cycle in $\Hom _B (N,N(-|X|-1))$. Hence, there is $\gamma \in \E$ with $|\gamma|=-|X|$ such that  $j_X(\partial) = [  \partial, \gamma]$.
Setting  $\Delta = j_X - ad (\gamma)$, by Lemma~\ref{lem20200318s} we have $\Delta \in \Der _A(\D)$, $\Delta (X) =1$, and $\Delta (\partial)=0$.
Since $j_X^2=0$, by Lemma~\ref{useful lemma} we can see that  $\Delta ^2 = ad  (\alpha )$, where  $\alpha = j_X ( \gamma ) + \gamma ^2$.
Now the assertion follows from Proposition~\ref{odd weak}.
\end{proof}

\section{Na\"ive liftability of DG modules}\label{sec20201110a}

In this section, we give a new characterization of (weak) lifting property of DG modules along simple extensions of DG algebras; see Theorem~\ref{naive thm one variable} below.

The notation used in this section comes from the previous sections.

\begin{para}\label{para20201110b}
If  $(N, \partial ^N)$  is a semifree DG $B$-module with a graded basis $\B$ as a free $B$-module, then it follows from Lemma~\ref{j-operator properties}(c) that $[j_X^{\B}(\partial ^N),\partial ^N]=0$. Hence, $j_X^{\B}(\partial ^N)$ defines the cohomology class  $[ j_X^{\B}(\partial ^N)]$ in $\Ext _B ^{|X|+1}(N, N)$.
\end{para}

\begin{thm}\label{prop20201110b}
Let $(N, \partial ^N)$  be a semifree DG $B$-module. Then the cohomology class $[ j_X^{\B}(\partial ^N)]$ is independent of the choice of the  $B$-free graded basis  $\B$.
\end{thm}

\begin{proof}
If  $\B '$  is another $B$-free graded basis of $N$, then by Proposition~\ref{automorphism} we have
$j_X^{\B} - j_X^{\B '}= ad ( \alpha )$,  for some $\alpha \in \E$.
Applying this equality to $\partial ^N$, we have the equality $j_X^{\B}(\partial ^N) - j_X^{\B'}(\partial ^N) = [\alpha,\partial ^N]$,
which means that the left-hand-side is null-homotopic as a DG module homomorphism $N \to N(-|X|-1)$.
Hence  $[j_X^{\B}(\partial ^N) ] = [j_X^{\B '}(\partial ^N)]$  as an element of $\Ext _B ^{|X|+1}(N, N)$.
\end{proof}

\begin{para}
Using the notation of Theorem~\ref{prop20201110b}, we simply denote the cohomology class $[ j_X ^{\B} ( \partial ^N) ]$ in  $\Ext _B ^{|X|+1}(N, N)$  by  $[ j_X ( \partial ^N) ]$.
\end{para}

\begin{para}\label{para20201110d}
Assume that $(N, \partial ^N)$ is a DG $B$-module, and let $N |_A$  denote $N$ regarded as a DG $A$-module via the natural map $A\to B$.
Note that  $N |_A \otimes _A B$  is a DG $B$-module with the differential
$\partial  (n \otimes b) = \partial ^N (n) \otimes b + (-1)^{|n|} n \otimes d^B b$,
for all $n \in N$ and $b \in B$.
If $N$ is a semifree DG $B$-module, then $N |_A$  is a semifree DG $A$-module.
Note also that there is a natural mapping $\pi _N\colon N |_A \otimes _A B \to N$ defined by $n \otimes b \mapsto nb$,
which is a (right) DG $B$-module epimorphism.
\end{para}

\begin{defn}\label{naive definition}
Assume that $N$ is a semifree (right) DG $B$-module. We say that $N$  is {\it na\"ively liftable}  to $A$ if
the map  $\pi _N$ from~\ref{para20201110d} is a split DG $B$-module epimorphism, i.e., there exists a DG $B$-module homomorphism $\rho\colon N \to  N |_A \otimes _A B$ that satisfies the equality $\pi _N \rho = \id _N$. Equivalently, $N$  is na\"ively liftable to $A$ if $\pi_N$ has a right inverse in the abelian category of right DG $B$-modules.
\end{defn}

\begin{para}\label{naive-odd}
In case that  $|X|$  is odd,
Nasseh and Yoshino~\cite[Theorem 3.6]{nassehyoshino} showed that there is a short exact sequence
\begin{equation}\label{eq20201110a}
0 \to N ( -|X| ) \to N |_A \otimes _A B  \xra{\pi _N} N \to 0
\end{equation}
of DG $B$-modules, and $N$ is weakly liftable to $A$ if and only if~\eqref{eq20201110a} is split.
Thus, in this case, na\"ive  liftability is equivalent to weakly liftability.

In contrast to the previous case, if  $|X|$  is even, then $B = A \oplus XA \oplus X^{(2)}A \oplus \cdots$  as an underlying  right $A$-module. Hence, if $\{ e_{\lambda}\}_{\lambda\in \Lambda}$ is a basis for $N$ as a free $B$-module, then
$\{ e_{\lambda} X ^{(i)} \otimes 1 \}_{i \in \mathbb{N}, \lambda\in \Lambda}$ is a basis for $N|_A \otimes _A B$ as a free $B$-module.
Therefore, one has
$$
N|_A \otimes _A B = \bigoplus _{i \in \mathbb{N}} \left(  \bigoplus _{\lambda\in \Lambda} e_{\lambda} X ^{(i)} \otimes B \right)  \cong \bigoplus _{i  \in \mathbb{N}}  N(-i |X|)
$$
as an underlying right $B$-module, and $\pi _N$ maps $w \in N (-i|X|)$ to $wX^{(i)} \in N$.
Note that $N|_A \otimes _A B$ by its nature is an infinite direct sum of shifts of $N$. Hence, the notion of na\"ive lifting is not equivalent to that of weak lifting in the sense of Definition~\ref{weak lifting defn} in the case where $|X|$ is even.
\end{para}

\begin{lem}\label{naive-even}
Assume  that $|X|$  is even.
If  a semifree DG $B$-module  $N$ is liftable to $A$, then  $N$  is na\"ively liftable to $A$.
\end{lem}

\begin{proof}
Let  $M$  be a DG $A$-module with $N \cong M \otimes _A B$.
Then, there is an isomorphism $N |_A \otimes _A  B \cong  M \otimes _A (B |_A \otimes _A B)$ of right DG $B$-modules such that
$$
\xymatrix{
N |_A \otimes _A  B\ar[rrr]^{\pi_N}\ar[d]_{\cong}&&&N\ar[d]^{\cong}\\
M \otimes _A (B |_A \otimes _A B)\ar[rrr]^{\id _M \otimes \pi_B}&&&M\otimes_A B
}
$$
is a commutative diagram.
Thus, if  $\pi _B$ has a right inverse, then so does $\pi_N$.

Define $\rho\colon B \to B|_A \otimes _A B$ by $\rho (b) = 1 \otimes b$, for all $b \in B$. Then,
$\rho$ is a right DG $B$-module homomorphism such that $\pi_B \rho = \id _B$. Hence, $\pi _B$ has a right inverse, and therefore, $\pi_N$ is split. This means that $N$  is na\"ively liftable to $A$, as desired.
\end{proof}

Here is our new characterization of weak lifting property of DG modules along simple extensions of DG algebras.

\begin{thm}\label{naive thm one variable}
Assume that  $(N,\partial ^N)$  is  a bounded below semifree DG $B$-module.
Then the following conditions are equivalent.
\begin{enumerate}[\rm(i)]
\item
$N$  is na\"ively liftable to $A$.
\item
$N$ is liftable to $A$ if $|X|$ is even, and $N$  is weakly liftable to $A$ (in the sense of Proposition~\ref{odd weak}) if $|X|$ is odd.
\item
The cohomology class  $[j_X (\partial ^N)]$ vanishes in  $\Ext _B^{|X|+1}(N, N)$.
\end{enumerate}
\end{thm}

The proof of (iii)$\implies$(ii) can be found implicitly in~\cite[Theorem 3.6]{nassehyoshino} (where there is no mention of the notion of $j$-operators) and explicitly in~\cite[Theorem 4.7]{OY}. For the convenience of the reader, we give the proof below.

\begin{proof}
(iii)$\implies$(ii)
Since  $[j_X (\partial ^N)]=0$ in  $\Ext _B^{|X|+1}(N, N)$, there is an $\alpha \in \E$ with $|\alpha|=-|X|$ such that  $j_X (\partial ^N) = [\partial^N, \alpha]$.

In case that  $|X|$  is odd, let   $\Delta = j_X - ad(\alpha)$. Then we have $\Delta (X) =1$ and  $\Delta (\partial ^N) =0$.
By the proof of Corollary~\ref{cor20201111a}, we have  $\Delta ^2 = ad (j_X(\alpha) + \alpha ^2)$. Hence, it follows from Proposition~\ref{odd weak} that  $N$ is weakly liftable to $A$.

In case that  $|X|$  is even, let   $\Delta = j_X + ad(\alpha)$. Then we have  $\Delta (X^{(n)})= X^{(n-1)}$ for all $n \geq 1$  and  $\Delta (\partial ^N)=0$.
Therefore, it follows from Proposition~\ref{cor20200318f} that $N$ is liftable to $A$, as desired.

(ii)$\implies$(i) follows from~\ref{naive-odd} and Lemma~\ref{naive-even}.

(i)$\implies$(iii)
By definition, there exists a DG $B$-module $N'$ such that $N |_A \otimes _A B \cong N \oplus N'$ as DG $B$-modules, and the differential on  $N |_A \otimes _A B$  equals $\partial ^N \otimes _A B$.
Let $\B=\{ e_{\lambda}\}$  be  a graded basis of $N$ as a free $A$-module. Then  $\B':=\{ e_{\lambda} \otimes 1\}$  is a graded basis of $N|_A \otimes _A B$ as a free $B$-module.

If $\partial ^N(e _{\lambda}) = \sum _{\mu} e_{\mu}a_{\mu \lambda}$,  where  $a_{\mu \lambda} \in A$, then
$(\partial ^N \otimes _A B)(e _{\lambda}\otimes 1) = \sum _{\mu} (e_{\mu} \otimes 1) a_{\mu \lambda}$.
Since  $\frac{d}{d X} (a_{\mu \lambda}) =0$, we have $j_X^{\B'} (\partial ^N \otimes _A B) =0$.
In particular, $[j_X^{\B'} (\partial ^N \otimes _A B)]=0$ in $\Ext^{|X|+1}_B(N |_A \otimes _A B,N |_A \otimes _A B)$.
Since $N |_A \otimes _A B \cong N \oplus N'$ as DG $B$-modules, we conclude from Theorem~\ref{prop20201110b} that
$[ j_X (\partial ^{N \oplus N'})] =0$  as an element of
$$
\Ext _B^{|X|+1}(N \oplus N', N \oplus N') =
\left[\begin{matrix} \Ext _B^{|X|+1}(N, N) & \Ext _B^{|X|+1}(N, N')) \\
\Ext _B^{|X|+1}(N', N) & \Ext _B^{|X|+1}(N',  N')
\end{matrix}\right].
$$
Let $\mathcal{B} ^{\sharp}$ be the basis of $N \oplus N'$ as a free $B$-module that is the union of bases of $N$ and $N'$ as free $B$-modules. Since  $\partial ^{N \oplus N'} = \left[\begin{smallmatrix} \partial ^N & 0 \\ 0 & \partial ^{N'} \end{smallmatrix}\right]$, we obtain
$$
0=[j_X ^{\mathcal{B} ^{\sharp}} (\partial ^{N \oplus N'}) ]= \left[\begin{matrix} [j_X(\partial ^N)] & 0 \\ 0 & [j_X(\partial ^{N'})] \end{matrix}\right]
$$
in which the second equality comes from our discussion in~\ref{para20200427a}.
Hence, $ [j_X(\partial ^N)]=0$ as an element of $\Ext _B^{|X|+1}(N,  N)$.
 \end{proof}


\section*{Acknowledgments}
A part of the discussion on this work occurred at Okayama University in June-July, 2018 and at the Institute for Research in Fundamental Sciences (IPM) in June-July, 2019. The first author is extremely grateful to the Department of Mathematics at Okayama University for the full support and great hospitality. The first and third authors would like to thank IPM for providing the opportunity for them to meet and continue their previous mathematical conversations in person.

\providecommand{\bysame}{\leavevmode\hbox to3em{\hrulefill}\thinspace}
\providecommand{\MR}{\relax\ifhmode\unskip\space\fi MR }
\providecommand{\MRhref}[2]{%
  \href{http://www.ams.org/mathscinet-getitem?mr=#1}{#2}
}
\providecommand{\href}[2]{#2}

\end{document}